\documentclass{amsart}
\usepackage{graphicx,latexsym}
\usepackage{amssymb,amsbsy,amsmath,amsthm}

\usepackage{fancyhdr}

\usepackage{mathabx,calrsfs,times}
  \usepackage{paralist}
 \usepackage[colorlinks=true]{hyperref}
\hypersetup{urlcolor=blue, citecolor=red}

\newtheorem{theorem}{Theorem}[section]
\newtheorem{corollary}[theorem]{Corollary}

\newtheorem{lemma}[theorem]{Lemma}
\newtheorem{proposition}[theorem]{Proposition}

\theoremstyle{definition}

\newtheorem{remark}{Remark}

\newcommand{\ep}{\varepsilon}

\newcommand{\bbN}{\mathbf{N}}

\newcommand{\bbR}{\mathbf{R}}
\newcommand{\bbT}{\mathbf{T}}
\newcommand{\bbZ}{\mathbf{Z}}

\newcommand{\bfN}{\mathbf{N}}
\newcommand{\bfR}{\mathbf{R}}

\renewcommand{\S}{\Sigma}

\renewcommand{\a}{\alpha}
\renewcommand{\b}{\beta}
\newcommand{\g}{\gamma}

\newcommand{\s}{\sigma}
\renewcommand{\phi}{\varphi}

\newcommand{\ol}[1]{\overline{#1}}

\renewcommand{\to}{\longrightarrow}

\title[Shmerkin-Wu theorem]{A new dynamical proof of the Shmerkin--Wu theorem}

\author{Tim Austin}
\address{Department of Mathematics, University of California, Los Angeles, Los Angeles, CA 90095-1555, USA}

\email{tim@math.ucla.edu}

\begin{document}
\maketitle

\thispagestyle{empty}





\begin{abstract}
Let $a < b$ be multiplicatively independent integers, both at least $2$.  Let $A,B$ be closed subsets of $[0,1]$ that are forward invariant under multiplication by $a$, $b$ respectively, and let $C := A\times B$.  An old conjecture of Furstenberg asserted that any planar line $L$ not parallel to either axis must intersect $C$ in Hausdorff dimension at most $\max\{\dim C,1\} - 1$.  Two recent works by Shmerkin and Wu have given two different proofs of this conjecture.  This note provides a third proof.  Like Wu's, it stays close to the ergodic theoretic machinery that Furstenberg introduced to study such questions, but it uses less substantial background from ergodic theory.  The same method is also used to re-prove a recent result of Yu about certain sequences of sums.
\end{abstract}

\section{Introduction}

Let $\bbT := \bbR/\bbZ$.  For any integer $a \ge 2$, we write $S_a$ for either of two transformations:
\begin{itemize}
\item $x\mapsto ax$ on $\bbT$;
\item $x\mapsto \{ax\}$ on $[0,1)$, where $\{\cdot\}$ denotes fractional part.
\end{itemize}
The obvious bijection $[0,1)\to \bbT$ identifies these two transformations, justifying the use of a single notation.  We sometimes leave the correct choice of domain to the context.  Similarly, for any $u \in \bbT$, we write $R_u$ for the rotation of $\bbT$ by $u$, or for the corresponding transformation of $[0,1)$.

Now let $a < b$ be multiplicatively independent integers, both at least $2$.  Let $A,B$ be closed subsets of $[0,1]$ that are forward invariant under $S_a$, $S_b$ respectively, and let $C := A\times B$.  Two recent papers, by Shmerkin~\cite{Shm19} and Wu~\cite{Wu19}, independently prove the following conjecture of Furstenberg~\cite{Fur70}.

\vspace{7pt}

\noindent\textbf{Theorem A} (Shmerkin--Wu theorem)\textbf{.}\ \ \emph{If $L$ is any line not parallel to either coordinate axis, then
\[\dim(C\cap L) \le \max\{\dim C,1\} - 1.\]}

\vspace{7pt}

Here and in the rest of the paper, $\dim$ denotes Hausdorff dimension. In a few places we also need upper box dimension, which is denoted by $\ol{\dim}_{\mathrm{B}}$.

Theorem A strengthens the classical slicing theorem of Marstrand, which provides the same inequality for \emph{almost} all lines when $C$ is any planar set. Furstenberg's original formulation of Theorem A concerns intersections of affine images of the sets $A$ and $B$, and for this reason it is often called Furstenberg's intersection conjecture.  It is equivalent to the above simply by writing the intersection $C\cap L$ in coordinates, as he already explains in his paper.

Both Shmerkin's and Wu's proofs of Theorem A use ergodic theory, but in very different ways.  Shmerkin's approach is quite quantitative, and sets aside most of Furstenberg's machinery from~\cite{Fur70}.  Shmerkin's main results have several other applications besides Theorem A.  Wu's work stays closer to Furstenberg's methods, but finds a way to use major additional results from abstract ergodic theory, particularly the unilateral Sinai factor theorem.

In the present note we offer a new proof of Theorem A.  It takes as its starting point one of Furstenberg's original results (Theorem~\ref{thm:Fur} below), and then uses different background from ergodic theory from Wu's.

Section~\ref{sec:Hdim} of this paper is devoted to some background results from geometric measure theory. Section~\ref{sec:Yu} gives a new proof of a recent theorem of Yu~\cite{HanYu20}.  This is not needed for the proof of Theorem A, but it offers a simple model setting to illustrate our approach.  Finally, Section~\ref{sec:A} gives the new proof of Theorem A.

\section{Some preliminaries on Hausdorff measure and dimension}\label{sec:Hdim}

Let $X$ be a metric space.  For $d \ge 0$, $\ep > 0$, and $A \subset X$, we write
\[m_d^{(\ep)}A := \inf\Big\{\sum_n (\mathrm{diam}\,F_n)^d:\ \langle F_n\rangle\ \hbox{a covering of}\ A\ \hbox{by sets of diam}\ < \ep\Big\}.\]
Then $m_d^\ast A :=\lim_{\ep \downarrow 0}m_d^{(\ep)}A$ is the \textbf{$d$-dimensional Hausdorff outer measure} of $A$.  Its restriction to the Borel sets defines a true measure, the \textbf{$d$-dimensional Hausdorff measure} $m_d$.  The properties of these measures and the resulting notion of Hausdorff dimension can be found in standard treatments such as~\cite{Fal14}.

\subsection{A generalized Marstrand theorem}\label{sec:Mar}

In the proofs below we use a generalization of Marstrand's classical slicing theorem.  I believe this generalization is widely known, but I do not know of a reference that contains the exact version we need, so I include a full proof here.  Our version is similar to~\cite[Theorem 8.1]{Hoc--notes}, but we give a more classical proof: compare, for instance,~\cite[Corollary 7.10]{Fal14}.

\begin{proposition}\label{prop:Mar}
Let $X$ and $Y$ be metric spaces and let $\mu$ be a finite Borel measure on $Y$.  Equip $X\times Y$ with the max-metric.  Assume there exists $a > 0$ such that
\[\mu B(y,r) \le r^{a - o(1)} \quad \hbox{as}\ r\downarrow 0 \quad \hbox{for}\ \mu\hbox{-a.e.}\ y,\]
where $B(y,r)$ is the open ball of radius $r$ around $y$, and where the rate implied by the notation $o(1)$ may depend on the point $y$.

Let $W \subset X\times Y$, and let $W_y := \{x:\ (x,y) \in W\}$ for each $y \in Y$. Then
\[\dim W_y + a \le \max\{\dim W,a\} \quad \hbox{for}\ \mu\hbox{-a.e.}\ y.\]
\end{proposition}

\begin{proof}
\emph{Step 1.}\quad We first complete the proof under a stronger assumption: that there exists $r_0 > 0$ such that $\mu B(y,r) \le r^a$ whenever $y \in Y$ and $0 < r \le r_0$.
Let $d := \dim W$, and assume this is finite, for otherwise the result is trivial.

Now let $d' > \max\{d,a\}$, let $\ep < r_0$, and let $\langle F_n\rangle$ be a covering of $W$ by sets of diameter less than $\ep$ such that
\[\sum_n (\mathrm{diam}\,F_n)^{d'} \le \ep.\]

For each $n$, let $G_n := \ol{\{y:\ (x,y) \in F_n\ \hbox{for some}\ x\in X\}}$.  Then
\[W_y\times \{y\} \subset \bigcup_{n:\ G_n \ni y}F_n\]
for each $y \in Y$, and hence
\[m_{d'-a}^{(\ep)}(W_y) \le \sum_n1_{G_n}(y)\cdot (\mathrm{diam}\ F_n)^{d' - a}.\]
Now we integrate against $d\mu(y)$. Since $\mathrm{diam}\,G_n \le \mathrm{diam}\,F_n \le r_0$ for each $n$, our strengthened assumption on $\mu$ gives
\begin{align*}
\ol{\int} m_{d'-a}^{(\ep)}(W_y)\,d\mu(y) &\le \sum_{n}\mu G_n\cdot (\mathrm{diam}\ F_n)^{d' - a} \\
&\le \sum_{n}(\mathrm{diam}\, G_n)^a\cdot (\mathrm{diam}\ F_n)^{d' - a}\\
&\le \sum_n(\mathrm{diam}\ F_n)^{d'} \le \ep.
\end{align*}
Here we use the upper Lebesgue integral $\ol{\int}$ because we do not know that $m_{d'-a}^{(\ep)} (W_y)$ is a measurable function of $y$. Since $\ep > 0$ was arbitrary, we conclude the following for $\mu$-a.e. $y$:
\[m_{d'-a} (W_y) = 0 \quad \forall d' > \max\{d,a\}, \quad \hbox{and hence} \quad \dim W_y \le \max\{d,a\}-a.\]

\vspace{7pt}

\emph{Step 2.}\quad Now consider $\mu$ satisfying the original hypothesis.  Let $a' \in (0,a)$, and let
\[Y_m := \{y \in Y:\ \mu B(y,r) \le r^{a'}\ \forall r < 1/m\} \quad \hbox{for each}\ m \in \bfN.\]
Then $\mu(\bigcup_m Y_m) = 1$, so it suffices to prove the result when $y \in Y_m$ for some $m$. But for this purpose we may replace $Y$ with $Y_m$, $W$ with $W\cap (X\times Y_m)$, and $\mu$ with $\mu(\,\cdot\,\cap Y_m)$.  This returns us to the stronger hypothesis of Step 1, except with $a'$ in place of $a$.  So now we conclude from Step 1 that
\[\dim W_y + a' \le \max\{\dim W,a'\}\]
for a $\mu$-a.e. $y$.  Now let $a'$ increase through a sequence of values to $a$.
\end{proof}

\subsection{Adic versions of the mass distribution principle}\label{sec:MDP}

The mass distribution principle is one of the most basic and versatile sources of lower bounds on Hausdorff measure and dimension.  Standard formulations can be found in~\cite[Section 4.1]{Fal14} or~\cite[Proposition 4.2]{Hoc--notes}.

For one- and two-dimensional Euclidean subsets, we need some versions of this principle with adic intervals taking the place of balls.  Such variants are well-known, but we include the precise statements we need for completeness.

Let $a$ and $n$ be integers with $a\ge 2$.  An \textbf{$a$-adic interval of depth $n$} is a real interval of the form $[ka^{-n},(k+1)a^{-n})$ for some $k\in \bbZ$. For any $x \in \bfR$, let $I_n(x)$ denote the $a$-adic interval of depth $n$ that contains $x$.  If we fix $n \ge 0$ and let $x$ vary in $[0,1)$, then the sets $I_n(x)$ constitute a partition of $[0,1)$.  These partitions become finer as $n$ increases.  They can serve as substitutes for centred intervals in the mass distribution principle:

\begin{lemma}\label{lem:fromHoc}
Let $\mu$ be a finite Borel measure on $[0,1)$, and let $d \ge 0$.  Then the following are equivalent:
\begin{enumerate}
\item[a.] $\mu (z-r,z+r) \le r^{d - o(1)}$ as $r\downarrow 0$ for $\mu$-a.e. $z$, where the rate implied by the notation $o(1)$ may depend on $z$;
\item[b.] $\mu I_n(z) \le a^{-dn + o(n)}$ as $n\to\infty$ for $\mu$-a.e. $z$, where the rate implied by the notation $o(n)$ may depend on $z$.
\end{enumerate}
If these conditions hold, then any Borel set with positive $\mu$-measure has Hausdorff dimension at least $d$. \qed
\end{lemma}

The equivalence of (a) and (b) is a special case of~\cite[Proposition 6.21]{Hoc--notes}, and the final implication is the mass distribution principle.

Now let $b\ge 2$ be another integer, and let $J_n(y)$ denote the $b$-adic interval of depth $n$ that contains a point $y \in \bbR$.  Let $u := \log a/\log b$.  Then, for each $n \in \bbN$ and $(x,y) \in \bfR^2$, let
\begin{equation}\label{eq:Dn}
Q_n(x,y) := I_n(x)\times J_m(y), \quad \hbox{where}\ m = \lfloor u n\rfloor.
\end{equation}
This choice of $m$ is the largest integer for which $b^{-m} \ge a^{-n}$. As a result, the rectangle $Q_n(x,y)$ is always close to being square: its height is at least its width, but no more than $b$ times its width.

If we fix $n$ and let $(x,y)$ vary in $[0,1)^2$, then the sets $Q_n(x,y)$ constitute a partition of $[0,1)^2$.  These partitions become finer as $n$ increases.  They participate in a two-dimensional variant of the mass distribution principle:

\begin{lemma}\label{lem:fromHoc2}
Let $\mu$ be a finite Borel measure on $[0,1)^2$, and let $d \ge 0$.  Then the following are equivalent:
\begin{enumerate}
\item[a.] $\mu B(z,r) \le r^{d - o(1)}$ as $r\downarrow 0$ for $\mu$-a.e. $z$, where the rate implied by the notation $o(1)$ may depend on $z$;
\item[b.] $\mu Q_n(z) \le a^{-dn + o(n)}$ as $n\to\infty$ for $\mu$-a.e. $z$, where the rate implied by the notation $o(n)$ may depend on $z$.
\end{enumerate}
If these conditions hold, then any Borel set with positive $\mu$-measure has Hausdorff dimension at least $d$. \qed
\end{lemma}

This time the equivalence of (a) and (b) is not quite a special case of~\cite[Proposition 6.21]{Hoc--notes}, because that reference concerns true $b$-adic squares.  However, the proof requires only that the rectangles $Q_n(x,y)$ have uniformly bounded aspect ratios, so it carries over without change.

\section{A theorem of Yu}\label{sec:Yu}

Before approaching Theorem A, we give a new proof of a recent theorem of Yu~\cite{HanYu20}:

\begin{theorem}\label{thm:Yu}
Let $u \in \bbT$ be irrational, let $a \ge 2$ be an integer, and let $v \in \bbT$ be arbitrary.  Then the closure of the set $\{nu + a^nv:\ n\in\bbN_0\}$ has Hausdorff dimension $1$.
\end{theorem}

Yu proves this result by an adaptation of Wu's chief innovation in~\cite{Wu19}, which is a certain application of the unilateral Sinai factor theorem referred to as the `Bernoulli decomposition method'.  Our next topic is a new, shorter proof of Theorem~\ref{thm:Yu} which avoids this machinery.  We include it here as a warmup to the coming proof of Theorem A.

As remarked in his paper, Yu's method really proves the following.  Let $\s:\bbT\times \bbT\to \bbT$ be the map $(u,v)\mapsto u+v$.

\begin{theorem}\label{thm:truYu}
If $C \subset \bbT\times \bbT$ is nonempty, closed, and forward invariant under $R_u\times S_a$, then $\dim \s[C] = 1$.
\end{theorem}

This implies Theorem~\ref{thm:Yu} by letting $C$ be the forward orbit closure of the point $(0,v)$ under $R_u \times S_a$.

Let $\pi_i$ for $i=1,2$ denote the first and second coordinate projections $\bbT\times \bbT\to \bbT$.  Since $C$ is closed and forward invariant under $R_u \times S_a$, its image $\pi_2[C]$ is closed and forward invariant under $S_a$.  Let $h$ be the topological entropy of the topological dynamical system $(\pi_2[C],S_a)$. A classical calculation of Furstenberg~\cite[Proposition III.1]{Fur67} gives
\[\dim \pi_2 [C] = \ol{\dim}_{\mathrm{B}} \pi_2[C] = \frac{h}{\log a}.\]
We denote this value by $d$ in the remainder of this section.

\begin{lemma}\label{lem:good-meas}
The set $C$ carries a Borel probability measure $\mu$ which is invariant and ergodic under $R_u\times S_a$ and such that $(\bbT,\pi_{2\ast}\mu,S_a)$ has Kolmogorov--Sinai entropy equal to $h$.
\end{lemma}

\begin{proof}
Since $S_a$ is expansive, the topological dynamical system $(\pi_2[C],S_a)$ has a measure of maximal entropy $\nu$, whose Kolmogorov--Sinai entropy therefore equals $h$.  Replacing $\nu$ with one of its ergodic components if necessary, we may assume it is ergodic.  Let $\mu$ be any $(R_u\times S_a)$-invariant and ergodic lift of this measure to $C$.
\end{proof}

We fix the measures $\mu$ and $\nu = \pi_{2\ast}\mu$ for the rest of this section.  The image of $\mu$ under $\pi_1$ is an invariant measure for $R_u$, so must equal Lebesgue measure $m$, by unique ergodicity.  In addition, let
\[\mu = \int_\bbT \delta_t \times \nu_t\,dt\]
be the disintegration of $\mu$ over $\pi_1$.

Since $(\bbT,m,R_u)$ has entropy zero, $h$ is also equal to
\begin{itemize}
\item the Kolmogorov--Sinai entropy of $(\bbT^2,\mu,R_u\times S_a)$, and
\item the relative Kolmogorov--Sinai entropy of $(\bbT^2,\mu,R_u\times S_a)$ over $\pi_1$ (see~\cite[Subsection 2.3]{EinLinWar--draft} for the definition of relative Kolmogorov--Sinai entropy over a factor, noting that those authors call it `conditional entropy').
\end{itemize}

Let $\S_a := {\{0,1,\dots,a-1\}}$, and let $\a:\bbT\to \S_\a$ be the map such that $\a(x)$ is the first digit in the $a$-ary expansion of $x$.  The sequence of observables $\a_n := \a\circ S_a^n$, $n= 0,1,2,\dots$, generates the whole Borel sigma-algebra of $\bbT$, so by the Kolmogorov--Sinai theorem the resulting partitions of $\bbT$ witness the full Kolmogorov--Sinai entropy of the system $(\pi_2[C],\nu,S_a)$.

By ergodicity and the relative Shannon--McMillan--Breiman theorem~\cite[Theorem 3.2]{EinLinWar--draft}, these entropy values imply that
\begin{equation}\label{eq:relSMB1}
\nu_t I_n(x) = e^{-hn + o(n)}  = (a^{-n})^{d+ o(1)} \quad \hbox{as}\ n\to\infty
\end{equation}
for $m$-a.e. $t$ and then for $\nu_t$-a.e. $x$, where the rates implied by the little-$o$ notation may depend on $t$ and $x$.  Note that~\cite[Theorem 3.2]{EinLinWar--draft} may be applied here because the factor system is simply a circle rotation, hence invertible, even though the full system $(\bbT^2,\mu,R_u\times S_a)$ is typically not invertible.

The asymptotic~\eqref{eq:relSMB1} is the starting point of our dimension estimates, via Lemma~\ref{lem:fromHoc}.

\begin{proof}[Proof of Theorem~\ref{thm:truYu}]
Let $C_t = \{y \in \bbT:\ (t,y) \in C\}$, and observe that
\[1 = \mu C = \int_\bbT \nu_t C_t\,dt,\]
and so $\nu_t C_t = 1$ for a.e. $t$. Therefore $\dim C_t \ge d$ for a.e. $t$, by~\eqref{eq:relSMB1} and Lemma~\ref{lem:fromHoc}.

Since $\pi_1[C]$ is nonempty, closed, and forward invariant under $R_u$, it must be the whole of $\bbT$ by minimality.  This implies that $\dim C \ge \dim \bbT = 1$, and so Proposition~\ref{prop:Mar} gives
\[\dim C_t +1 \le \dim C  \quad \hbox{for}\ m\hbox{-a.e.}\ t.\]
Combining the properties above, we find that a.e. $t$ is a witness to the inequality
\begin{equation}\label{eq:Yu1}
d +1 \le \dim C.
\end{equation}

Next, consider the map
\[(\s, \pi_2):\bbT\times \bbT\to \bbT\times \bbT : (x,y)\mapsto (x+y,y).\]
This is bi-Lipschitz, so it preserves the dimensions of subsets, and we obtain
\begin{equation}\label{eq:Yu2}
\dim C = \dim \big((\s, \pi_2)[C]\big) \le \dim (\s[C]\times \pi_2[C]) \le \dim \s[C] + d,
\end{equation}
where the last inequality is a standard bound for the Hausdorff dimension of a product~\cite[Product formula 7.3]{Fal14}, combined with the fact that $\ol{\dim}_{\mathrm{B}} \pi_2[C] = d$.  Concatening the inequalities~\eqref{eq:Yu1} and~\eqref{eq:Yu2}, the terms involving $d$ cancel to leave $\dim \s[C] \ge 1$.
\end{proof}

\begin{remark}
The set $\s[C]$ may miss a large portion of the circle $\bbT$, even though it has full dimension.  Indeed, if $\langle u_n\rangle$ is any sequence in $\bbT$, then $v \in \bbT$ may be chosen one decimal digit at a time so that the number $u_n + 10^n v$ lies in $[0,2/10]$ modulo $1$ for every $n$.  Then the closure of the set $\{u_n + 10^n v:\ n \in \bbN_0\}$ is entirely contained in this small interval.  I do not know whether the closure of that set can always be made to have Lebesgue measure zero.
\end{remark}

\section{Proof of Theorem A}\label{sec:A}

We now return to the setting of the Introduction.

\subsection{Furstenberg's auxiliary transformation}

Let $u := \log a/\log b$ as before.  This is irrational since $a$ and $b$ are multiplicatively independent.  Ergodic theory enters the proof of Theorem A because of the following specific transformation.  On the space $X = C\times \bbT$, we define
\[T(x,y,t) := \left\{\begin{array}{ll}(S_ax,S_by,R_{1-u} t) &\quad \hbox{if}\ 0 \le t < u\\
(S_ax,y,R_{1-u} t) &\quad \hbox{if}\ u \le t < 1. \end{array}\right.\]
This is a skew product over the irrational circle rotation $R_{1-u}$.  Upon iteration, it yields
\[T^n(x,y,t) = \big(S^n_ax,S^{m(t,n)}_by,R_{n(1-u)} t\big),\]
where $m(t,n)$ is the number among the points
\[t,\ t+(1-u),\ t + 2(1-u),\ \dots,\ t+(n-1)(1-u)\]
that lie in the interval $[0,u)$ mod $1$.  In the sequel we need the estimate
\begin{equation}\label{eq:sum-est}
|m(t,n) - u n| \le 1 \quad \forall t,n,
\end{equation}
which is~\cite[Lemma 7]{Fur70}.  These constructs are discussed further in~\cite[Section 5]{Wu19}.

Given $z \in \bbR^2$ and $t \in \bbR$, let $L_{z,t}$ be the planar line that has slope $b^t$ and passes through $z$.  The next theorem is a corollary of~\cite[Theorem 9]{Fur70}, which is the technical heart of that paper.  It is a kind of analog of Lemma~\ref{lem:good-meas} for the purposes of this section, but it is much more substantial than that lemma.

\begin{theorem}\label{thm:Fur}
If any planar line, not parallel to either coordinate axis, intersects $C$ in dimension at least $\g$, then there is a probability measure $\ol{\mu}$ on $X$ which is invariant and ergodic for $T$ and such that
\[\dim (C\cap L_{z,t}) \ge \g \quad \hbox{for $\ol{\mu}$-a.e.}\ (z,t).\]
\end{theorem}

In the remainder of the paper we use the invariant measure $\ol{\mu}$ promised by this theorem to prove that $\g$ is at most $\max\{\dim C,1\} -1$.

The image of $\ol{\mu}$ under the final coordinate projection to $\bbT$ is invariant under the irrational rotation $R_{1-u}$, so it must be Lebesgue measure $m$, by unique ergodicity.  Let
\begin{equation}\label{eq:disint}
\ol{\mu} = \int_\bbT \mu_t \times \delta_t\,dt
\end{equation}
be the disintegration of $\ol{\mu}$ over that projection to $\bbT$, and let $\mu = \int \mu_t\,dt$ be the projection of $\ol{\mu}$ to $C$.

Let $h$ be the Kolmogorov--Sinai entropy of the system $(X,\ol{\mu},T)$, and let $d := h/\log a$.  Then $h$ is also the relative entropy of $(X,\ol{\mu},T)$ over the coordinate projection to $(\bbT,m,R_{1-u})$, since the latter system has entropy zero.

\subsection{Consequences of the Shannon--McMillan--Breiman theorem}

Let ${\S_c := \{0,1,\dots,c-1\}}$ for any positive integer $c$.  Let $\a:X\to \S_a$ be the map such that $\a(x,y,t)$ is the first digit in the $a$-ary expension of $x$, and let $\b:X\to \S_b$ be the analogous map which reports the first digit in the $b$-ary expansion of $y$.  For each integer $n \ge 0$, let $\a_n := \a\circ T^n$ and $\b_n := \b\circ T^n$, and let
\[\a_{[0;n]}(z,t) := (\a_0(z,t),\dots,\a_n(z,t)) \quad \hbox{and} \quad \b_{[0;n]}(z,t) := (\b_0(z,t),\dots,\b_n(z,t)).\]
In this notation, $\a_{[0;n-1]}(x,y,t)$ lists the first $n$ digits in the $a$-ary expansion of $x$.  Similarly, $\b_{[0;n-1]}(x,y,t)$ lists the first $m(t,n)$ digits in the $b$-ary expansion of $y$, but with repetitions so that the output is a sequence of length $n$.  See the discussion preceding~\cite[Lemma 7]{Fur70}.

For each $n\ge 1$, the level sets of the combined map $(\a_{[0;n]},\b_{[0;n]})$ constitute a partition of $X$.  We write $P_n(z,t)$ for the cell of this partition that contains $(z,t)$: that is,
\[P_n(z,t) := \big\{(z',t'):\ \a_{[0;n]}(z',t') = \a_{[0;n]}(z,t)\ \hbox{and}\ \b_{[0;n]}(z',t') = \b_{[0;n]}(z,t)\big\}.\]
Together with the coordinate projection from $(X,\ol{\mu},T)$ to the zero-entropy system $(\bbT,m,R_{1-u})$, these finer and finer partitions $P_n$ generate the whole Borel sigma-algebra of $X$.  Therefore, by the Kolmogorov--Sinai theorem~\cite[Theorem 2.20]{EinLinWar--draft}, the entropy rate of the process $(\a_0,\b_0)$, $(\a_1,\b_1)$, \dots that generates these partitions is equal to the full Kolmogorov--Sinai entropy $h$ of the system ${(X,\ol{\mu},T)}$.  This fact is needed in the proofs of Propositions~\ref{prop:reg1} and~\ref{prop:reg2} below, where the measures of cells of $P_n$ are estimated in terms of $h$ using the Shannon--McMillan--Breiman theorem.

The next lemma compares the cells $P_n$ with other, simpler subsets of $X$.  To formulate it, we endow $\bbR^2$ with the max-metric, and let $B(z,r)$ denote the open ball of radius $r$ around $z \in \bbR^2$ in this metric.

\begin{lemma}\label{lem:ball-in-ball}
There is a fixed positive integer $c$ such that
\[Q_{n+c}(x,y) \times \{t\} \subset P_n(x,y,t) \subset B((x,y),ba^{-n})\times \bbT\]
for all $(x,y) \in C$, $t \in \bbT$, and $n \in \bbN$, where $Q_n$ is the notation from~\eqref{eq:Dn}.
\end{lemma}

\begin{proof}
Choose $c$ large enough that $u c \ge 2$, and suppose that $(x',y') \in Q_{n+c}(x,y)$.  Then $x'$ and $x$ agree in the first $n+c$ digits of their $a$-adic expansions, and $y'$ and $y$ agree in the first $m$ digits of their $b$-adic expansions, where $m = \lfloor u(n+c)\rfloor$.  By our choice of $c$, this satisfies
\[m = \lfloor u n + u c\rfloor \ge \lfloor u n\rfloor + 2 \ge u n + 1.\]
In view of~\eqref{eq:sum-est}, this implies that $m\ge m(t,n)$, and so $\a_{[0;n]}$ and $\b_{[0;n]}$ both agree on the two inputs $(x,y,t)$ and $(x',y',t)$.   This proves the first inclusion.

For the second inclusion, observe from~\eqref{eq:sum-est} that if $(x',y',t') \in P_n(x,y,t)$ then $x'$, $x$ must lie within distance $a^{-n}$ and $y'$, $y$ must lie within distance $b^{-u n + 1} = ba^{-n}$.  Hence $(x',y')$ must lie within distance $ba^{-n}$ of $(x,y)$.  This proves the second inclusion.
\end{proof}

We are now ready to derive estimates on the local behaviour of various measures on $C$: first the projection $\mu$, and then the disintegrand $\mu_t$ in~\eqref{eq:disint} for typical values of $t$.

\begin{proposition}\label{prop:reg1}
The measure $\mu$ has the property that
\begin{equation}\label{eq:bigballs}
\mu B(z,r) \ge r^{d + o(1)} \quad \hbox{as}\ r\downarrow 0\quad \hbox{for}\ \mu\hbox{-a.e.}\ z,
\end{equation}
where the rate implied by the notation $o(1)$ may depend on the point $z$.
\end{proposition}

\begin{proof}
Equivalently, we must show that
\[\ol{\mu}\big(B(z,r)\times \bbT\big) \ge r^{d + o(1)} \quad \hbox{as}\ r\downarrow 0\quad \hbox{for}\ \ol{\mu}\hbox{-a.e.}\ (z,t).\]
By the second inclusion of Lemma~\ref{lem:ball-in-ball}, the measure on the left is bounded below by $\ol{\mu}P_n(z,t)$, where $n$ is the smallest integer for which $ba^{-n}\le r$. By the Shannon--McMillan--Breiman theorem, we have
\[\ol{\mu}P_n(z,t) = e^{-hn - o(n)} = a^{-dn - o(n)} = r^{d + o(1)} \quad \hbox{as}\ n\to\infty\]
for $\ol{\mu}$-a.e. $(z,t)$.
\end{proof}

\begin{corollary}\label{cor:reg1}
There is a Borel subset $C_0$ of $C$ such that $\mu C_0 = 1$ and $\dim  C_0 \le d$.
\end{corollary}

\begin{proof}
By Proposition~\ref{prop:reg1}, there is a Borel set $C_0$ such that~\eqref{eq:bigballs} holds for every $z \in C_0$.  Now a standard covering argument shows that $m_s C_0 = 0$ for all $s > d$, so $\dim C_0 \le d$: see, for instance,~\cite[Proposition 4.9(b)]{Fal14} or~\cite[Proposition 6.24]{Hoc--notes}.
\end{proof}

Next we turn to the disintegrands in~\eqref{eq:disint}.

\begin{proposition}\label{prop:reg2}
For $m$-a.e. $t$, the measure $\mu_t$ has the property that
\[\mu_t B(z,r) \le r^{d - o(1)} \quad \hbox{as}\ r\downarrow 0 \quad \hbox{for}\ \mu_t\hbox{-a.e.}\ z,\]
where the rate implied by the notation $o(1)$ may depend on the pair $(z,t)$.
\end{proposition}

\begin{proof}
By Lemma~\ref{lem:fromHoc2}, it is equivalent to show that
\[\mu_t Q_n(z) \le a^{-dn + o(n)} \quad \hbox{as}\ n\to\infty \quad \hbox{for}\ \mu_t\hbox{-a.e.}\ z.\]
This fact has a similar proof to Proposition~\ref{prop:reg1}, except that now we must use the relative Shannon--McMillan--Breiman theorem~\cite[Theorem 3.2]{EinLinWar--draft}.  Indeed, the first inclusion of Lemma~\ref{lem:ball-in-ball} gives
\[\mu_t Q_n(z) = (\mu_t\times \delta_t)(Q_n(z)\times \{t\}) \le (\mu_t\times \delta_t)(P_{n-c}(z,t))\]
for all sufficiently large $n$.  By the relative Shannon--McMillan--Breiman theorem, this right-hand side behaves as
\[e^{-h(n-c) + o(n)} = e^{-hn + o(n)} = a^{-dn + o(n)} \quad \hbox{as}\ n\to\infty\]
for $m$-a.e. $t$ and then for $\mu_t$-a.e. $z$.  (As in Section~\ref{sec:Yu}, we can apply~\cite[Theorem 3.2]{EinLinWar--draft} because the factor system is a circle rotation, hence invertible.)
\end{proof}

\subsection{Completion of the proof}\label{subs:A}

\begin{proof}[Completed proof of the Theorem A]
\emph{Step 1.}\quad From the product set $C_0\times C$, we construct the following new sets:
\begin{itemize}
\item For each $t \in [0,1)$, let $D_t := \big\{(z,w) \in C_0\times C:\ w \in L_{t,z}\setminus \{z\}\big\}$.
\item Let
\[D := \big\{(z,w) \in C_0\times C:\ w \in L_{z,t}\setminus\{z\}\ \hbox{for some}\ t \in [0,1)\big\} = \bigcup_{t \in [0,1)} D_t.\]
In prose, this is the set of pairs $(z,w)$ of distinct points in $C$ such that (i) $z$ lies in the subset $C_0$ and (ii) the segment from $z$ to $w$ has positive slope lying in $[1,b)$.
\item Finally, let
\[\ol{D} := \big\{(z,w,t) \in C_0\times C\times [0,1):\ w \in L_{z,t}\setminus\{z\}\big\}.\]
This is similar to $D$, except it records explicitly the slope between $z$ and $w$.
\end{itemize}
The next few steps of the proof estimate and relate the dimensions of these sets.

\vspace{7pt}

\emph{Step 2.} \quad By Furstenberg's calculation for multiplication-invariant sets in~\cite[Proposition III.1]{Fur67}, both $A$ and $B$ are exact dimensional: that is, each has equal Hausdorff and upper and lower box dimensions.  Using this fact, standard formulae for dimensions of products~\cite[Product formulae 7.2, 7.3]{Fal14} give (i) that $C$ is also exact dimensional, and then (ii) these estimates:
\begin{equation}\label{eq:dimD}
\dim  D \le \dim (C_0\times C) = \dim  C_0 + \dim  C \le d + \dim  C.
\end{equation}

\vspace{7pt}

\emph{Step 3.}\quad The slope between $z$ and $w$ is a Lipschitz function of the pair $(z,w)$ when we restrict to any set of the form $\{(z,w) \in D:\ |z-w| \ge 1/n\}$, $n\in\bfN$.  Therefore $\ol{D}$ is a countable union of Lipschitz images of subsets of $D$, and so $\dim \ol{D} \le \dim D$.  The reverse inequality here is immediate, so in fact $\dim \ol{D} = \dim D$.

\vspace{7pt}

\emph{Step 4.}\quad Applying Proposition~\ref{prop:Mar} to the set $\ol{D}$ and the projection $(z,w,t)\mapsto t$, we obtain
\begin{equation}\label{eq:dim-drop}
\dim  D_t + 1 \le \max\{\dim \ol{D},1\} = \max\{\dim D,1\} \quad \hbox{for}\ m\hbox{-a.e.}\ t.
\end{equation}
Combining this with our previous results, it follows that $m$-a.e. $t$ satisfies~\eqref{eq:dim-drop} and also:
\begin{itemize}
\item[i.] $\dim(C \cap L_{z,t}) \ge \g$ for $\mu_t$-a.e. $z$ (by Theorem~\ref{thm:Fur});
\item[ii.] $\mu_t B(z,r) \le r^{d - o(1)}$ as $r\downarrow 0$ for $\mu_t$-a.e. $z$ (by Proposition~\ref{prop:reg2});
\item[iii.] $\mu_t C_0 = 1$, where $C_0$ is the set provided by Corollary~\ref{cor:reg1} (in view of the disintegration~\eqref{eq:disint} and the fact that $\mu C_0 = 1$).
\end{itemize}
Fix such a value of $t$ for the rest of the proof.

\vspace{7pt}

\emph{Step 5.}\quad Consider the restricted coordinate projection
\[\pi:D_t\to C_0:(z,w)\mapsto z.\]
By property (iii) above, the measure $\mu_t$ is supported by the target $C_0$ of this map.
Under this map, the pre-image $\pi^{-1}\{z\}$ is precisely the set $\{z\}\times ((C\cap L_{z,t})\setminus \{z\})$, whose dimension is equal to $\dim (C\cap L_{z,t})$.  Therefore properties (i) and (ii) and another appeal to Proposition~\ref{prop:Mar} give
\begin{equation}\label{eq:dim-intersection}
\g + d \le \dim (C\cap L_{z,t}) + d  \le \max\{\dim  D_t,d\} \quad \hbox{for}\ \mu_t\hbox{-a.e.}\ z.
\end{equation}

If $\dim D_t \le d$, then~\eqref{eq:dim-intersection} implies at once that $\g = 0$, which complies with Theorem A.  On the other hand, if $\dim D_t > d$, then the right-hand side in~\eqref{eq:dim-intersection} equals $\dim D_t$.  In this case we concatenate inequalities~\eqref{eq:dimD},~\eqref{eq:dim-drop} and~\eqref{eq:dim-intersection} to obtain
\[\g + d + 1 \le \max\{d + \dim C,1\}.\]
If this maximum is equal to $1$, then $\g = d = 0$, and otherwise we can cancel $d$ to conclude that $\g \le \dim C - 1$.
\end{proof}

\section*{Acknowledgements}

I am grateful to Adam Lott for reading an early version of this manuscript and pointing out several corrections.

\medskip
\medskip

\end{document}